\documentclass{amsart}
\usepackage{graphicx}
\usepackage{enumerate}
\usepackage{amssymb}
\usepackage{amsaddr}
\newtheorem{thm}{Theorem}[section]

\newtheorem{lem}[thm]{Lemma}

\newtheorem{conj}[thm]{Conjecture}
\newtheorem{defn}[thm]{Definition}

\theoremstyle{definition}
\theoremstyle{remark}

\numberwithin{equation}{section}

\begin{document}

\title{On the Length of a Partial Independent Transversal in a Matroidal Latin Square}
\author{Daniel Kotlar and Ran Ziv}
\address{Computer Science Department, Tel-Hai College, Upper Galilee 12210, Israel}%
\email{dannykot@telhai.ac.il}

\subjclass{68R05, 05B15, 05B35, 15A03}
\keywords{Latin square, matroidal Latin square, partial independent transversal}

\begin{abstract}
We suggest and explore a matroidal version of the Brualdi - Ryser conjecture about Latin squares. We prove that any $n\times n$ matrix, whose rows and Columns are bases of a matroid, has an independent partial transversal of length $\lceil2n/3\rceil$. We show that for any $n$, there exists such a matrix with a maximal independent partial transversal of length at most $n-1$.
\end{abstract}
\maketitle
\section{Introduction}\label{section1}
A Latin Square of order $n$ is an $n\times n$ array $L$ with entries taken from the set $\{1,\ldots,n\}$, where each entry appears exactly once in each row or column of $L$. A \emph{partial transversal} of size $k$ of a Latin square $L$
is a subset of $k$ different entries of $L$, where no two of them lie in the same row or column.
\\
The maximal size of a partial transversal in $L$ will be denoted here by $t(L)$ and the minimal size of $t(L)$, over all Latin squares $L$ of order $n$, will be denoted by $T(n)$.
\\
It has been conjectured by Ryser \cite{Ryser67} that $T(n) = n$ for every odd $n$ and by Brualdi \cite{Denes74} (see also \cite{BruRys} p. 255) that $T(n) = n-1$ for every even $n$. Although these conjectures are still unsettled, a consistent progress has been made towards its resolution:  Koksma \cite{Koksma69} proved that for $n\geq 3$,
 $T(n)\geq \lceil(2n+1)/3\rceil$. This bound was improved by Drake \cite{Drake77} to $T(n) \geq  \lceil3n/4\rceil$  for $n>7$, and again by de Veris and Wieringa \cite{VriesWier} who obtained a lower bound of $\lceil (4n-3)/5 \rceil$ for $n\geq 12$. Woollbright \cite{Wool78} showed that $T(n)\geq\lceil n- \sqrt{n} \rceil$. A similar result was obtained independently by Brouwer, de Vries and Wieringa \cite{Brouwer78}. Recently, Hatami and Shor \cite{HatShor08} proved that $T(n)\geq n-O(\log^2n)$. See also a recent comprehensive survey by Wanless \cite{wan11}.
\\
The aim of this note is to suggest and explore a matroidal version of the Brualdi-Ryser conjectures. For basic texts on matroids the reader is referred to Welsh \cite{Welsh76}, Oxley \cite{Oxley11} and White \cite{White86}.
\begin{defn}\label{def1}
Let $(M,S)$ be a matroid $M$ on a ground set $S$. A matroidal Latin square (abbreviated MLS) of degree $n$ over $(M,S)$ is an $n\times n$  matrix $A$ whose entries are elements of $S$, where each row or column of $A$ is a base of $M$.
\end{defn}

Notice that a matroidal Latin square reduces to a Latin square if M is a partition matroid.
We mention that according to a well-known conjecture of Rota \cite{HuangRota94} every set of $n$ bases of a matroid of rank $n$ can be arranged to form an MLS of degree $n$ so that its rows consist of the original bases.
\begin{defn}\label{def2}
An independent partial transversal of an MLS $A$ is an independent subset of entries of $A$ where no two of them lie in the same row or column of $A$.
\end{defn}
We propose the following analogue of Brualdi's conjecture:
\begin{conj}\label{conj1}
Every MLS of degree $n$ has an independent partial transversal of size $n-1$.
\end{conj}
In view of Ryser's conjecture, it is natural to ask whether in Conjecture~\ref{conj1} an independent transversal of size $n$ exists whenever $n$ is odd. Theorem~\ref{thm2} asserts that this is not the case.

\section{A lower bound for a maximal independent partial transversal}\label{section2}
 Let $A=(a_{ij})_{i,j=1}^n$ be an MLS of degree $n$ over a matroid $M$. Let $T$ be an independent partial transversal of size $t$. Without loss of generality we may assume that the elements of $T$ are the first $t$ elements of the main diagonal of $A$. That is
\begin{equation}\label{eq:submatrices}
A=\left(\begin{array}{c|c}
  B & C \\ \hline
  D & E
  \end{array} \right)
\end{equation}
where $B$, $C$, $D$ and $E$ are sub-matrices of $A$ of dimensions $t\times t$, $t\times (n-t)$, $(n-t)\times t$ and $(n-t)\times (n-t)$ respectively, and $T$ constitutes the main diagonal of $B$. If $T$ is of maximal length, then $t\geq \lceil n/2\rceil$. Otherwise $dim(E)\geq n-t>t=dim(T)$ and thus $E$ would contain an element that is not spanned by $T$ and hence can be added to $T$, contradicting the maximality of $T$. In order to show that $t\geq\lceil 2n/3\rceil$ we shall need the following lemma:
\begin{lem}\label{lem1}
Let $X$ be a finite set and let $s>|X|/2$. Let $X_1,\ldots,X_s$ be a family of subsets of $X$, each of size at least $s$. Then there exists some $X_i$, all of whose elements appear in other subsets in the family.
\end{lem}
\begin{proof}
Let $Y_1$ be the set of elements in $X$ that appear in exactly one of the subsets $X_1,\ldots,X_s$ and let $Y_2$ be the set of elements in $X$ that appear in at least two of the subsets $X_1,\ldots,X_s$. Let $k_1=|Y_1|$ and $k_2=|Y_2|$. Assume, by contradiction, that each $X_i$ contains at least one element of $Y_1$. Then $k_1\geq s$ and thus
\begin{equation}\label{eq:1}
k_2\leq |X|-k_1\leq |X|-s<|X|/2
\end{equation}
(since $s>|X|/2$). If, for some $i$, $|X_i\cap Y_1|=1$ then $|X_i\cap Y_2|\geq s-1$ and thus $k_2\geq s-1 > |X|/2-1$. It follows that $k_2\geq |X|/2$, contradicting (\ref{eq:1}).
It follows that for all $i$, $|X_i\cap Y_1|\geq2$. Then $k_1\geq 2s$ and thus $k_2\leq |X|-k_1\leq |X|-2s < |X|-|X| =0$, which is absurd. This proves the lemma.
\end{proof}
\begin{thm}\label{thm1}
Let $A$ be an MLS of degree $n$ over a matroid $M$. Then $A$ contains an independent partial transversal of size $\lceil2n/3\rceil$.
\end{thm}
\begin{proof}
We use the notations from the beginning of Section~\ref{section2}. Since $T$ is maximal, all the elements in the sub-matrix $E$ are spanned by $T$. Let $T_E$ be the minimal subset of $T$ that spans $E$ (this set is unique since $T$ is independent.) Since $dim(E)\geq n-t$ then $|T_E|\geq n-t$ and thus $|T\setminus T_E|\leq t-(n-t)=2t-n$. Since each row of $A$ is a base and all the elements of $E$ are spanned by $T$, each row of the sub-matrix $D$ contains a subset of size $n-t$ that complement $T$ to a base. In particular, each row of $D$ contains at least $n-t$ elements that are not spanned by $T$. Let $X=\{1,\ldots,t\}$ be the set of indices of the columns of $D$. For each of the $n-t$ rows in $D$ we define a subset $X_i\subseteq X$, $i=t+1,\ldots,n$, in the following way: $j\in X_i$ if and only if the $j$th element of the $i$th row of $A$ is not spanned by $T$. It follows that $|X_i|\geq n-t$ for all $i=t+1,\ldots,n$. Now assume, by contradiction, that $t<2n/3$. Then $n-t>n/3>t/2$. So we have a set $X$ of size $t$ and $n-t$ subsets $X_{t+1},\ldots,X_n$, each of size at least $n-t$, such that $n-t>t/2$. Let $s=n-t$. By Lemma~\ref{lem1} we conclude that there exists a subset $X_i$ all of whose elements are contained in other subsets in the family $X_{t+1},\ldots,X_n$. This means that there is a row in $D$ containing at least $n-t$ elements that are not spanned by $T$ and for each such element there exists another element in the same column in $D$ that is not spanned by $T$. It follows that $D$ contains at least $n-t$ columns, each containing at least two elements that are not spanned by $T$. Since $t<2n/3$ we have that $|T\setminus T_E|\leq 2t-n<n/3<n-t$. So there exists $j\leq t$ such that (1) $a_{jj}\in T_E$ and (2) the $j$th column of $D$ contains at least two elements that are not spanned in $T$. Let $x\in E$ be such that its support (i.e., its minimal spanning set) in $T$ contains $a_{jj}$ and let $y$ and $z$ be two elements in the $j$th column of $D$ that are not spanned by $T$. We may assume that $x$ and $y$ are not in the same row (otherwise we take $z$ instead of $y$). Since $T\cup\{y\}$ is independent, and the support of $x$ in $T$ contains $a_{jj}$, it follows that $T\setminus\{a_{jj}\}\cup \{y\}$ does not span $x$ and thus $S\setminus\{a_{jj}\}\cup \{x,y\}$ is an independent partial transversal in $A$ of length $t+1$, contrary to the maximality of $T$. Thus $t$ must be at least $\lceil2n/3\rceil$.
\end{proof}
\section{An upper bound of size $n-1$ for an MLS of degree $n$}\label{section3}
It is well known that for any even $n$ there exist Latin squares of order $n$ with no transversal of size $n$. The following theorem shows that for any $n$ there exists an MLS of degree $n$ with no independent transversal of size $n$.
\begin{thm}\label{thm2}
Let $v_1,v_2,\dots,v_n$ be a basis of a vectorial matroid of rank $n$. Then the matrix $A=(a_{ij})_{i,j=1}^n$, whose elements are $a_{ii}=v_1$, for $i=1,\ldots,n$, and $a_{ij}=v_i-v_j$, for $1\leq i\ne j\leq n$, is an MLS of order $n$ with no independent transversal of size $n$.
\end{thm}
\begin{proof}
We leave it to the reader to check that the rows and columns of $A$ are independent. Let $T$ be a transversal of size $n$ in $A$. We show that $T$ is not independent. If $T$ does not contain elements of the main diagonal of $A$, then, since each row and column is represented exactly once among the elements of $T$, the sum of the elements of $T$ is 0, and $T$ is not independent. Thus we may assume that $T$ meets the main diagonal exactly once. Let $a_{ii}=v_1\in T$. If $i=1$ then the sum of the elements of $T-a_{11}$ is 0. If $i>1$, then $v_i$ is not spanned by $T$, so $T$ is not a basis, and thus, is not independent.
\end{proof}
\bibliographystyle{amsplain}
\bibliography{kzrefs5}
\end{document}